\documentclass[final,12pt]{article}
\usepackage{hyperref}
\usepackage{amssymb,amsmath,amsthm}
\usepackage{enumerate}
\usepackage[conditional,light,first,bottomafter]{draftcopy}
\draftcopyName{DRAFT\space\today}{130}
\draftcopySetScale{65}
\usepackage[letterpaper,hmargin=3.5cm,vmargin=3.5cm]{geometry}
\usepackage{setspace}
\makeatletter
\renewcommand{\section}{\@startsection%
{section}%
{1}%
{0em}%
{1.7em}%
{1.2em}%
{\normalfont\large\centering\bfseries}}
\renewcommand{\@seccntformat}[1]%
{\csname the#1\endcsname.\hspace{0.5em}}
\makeatother

\newtheorem{theorem}{Theorem}[section]

\newtheorem{lemma}{Lemma}[section]

\theoremstyle{definition}
\newtheorem{definition}{Definition}
\newtheorem{remark}{Remark}

\newcommand{\abs}[1]{\left|#1\right|}

\newcommand{\inner}[2]{\left\langle#1,#2\right\rangle}
\newcommand{\floor}[1]{\left\lfloor#1 \right\rfloor}

\newcommand{\nats}{\mathbb{N}}
\newcommand{\complex}{\mathbb{C}}
\newcommand{\integers}{\mathbb{Z}}
\newcommand{\cc}[1]{\overline{#1}}

\DeclareMathOperator{\rank}{rank}

\DeclareMathOperator*{\diag}{diag}
\begin{document}
\begin{titlepage}
\title{On a linear interpolation problem for $n$-dimensional vector polynomials
\footnotetext{%
Mathematics Subject Classification(2010):
30E05; 
41A05. 
}
\footnotetext{%
Keywords:
Linear interpolation;
Spaces of vector polynomials.
}
\hspace{-8mm}
\thanks{%
Research partially supported by UNAM-DGAPA-PAPIIT IN105414
}%
}

\author{
\textbf{Mikhail Kudryavtsev}
\\
\small Department of Mathematics\\[-1.6mm]
\small Institute for Low Temperature Physics and Engineering\\[-1.6mm]
\small Lenin Av. 47, 61103\\[-1.6mm]
\small Kharkov, Ukraine\\[-1.6mm]
\small\texttt{kudryavtsev@onet.com.ua}
\\[2mm]
\textbf{Sergio Palafox}
\\
\small Departamento de F\'{i}sica Matem\'{a}tica\\[-1.6mm]
\small Instituto de Investigaciones en Matem\'aticas Aplicadas y en Sistemas\\[-1.6mm]
\small Universidad Nacional Aut\'onoma de M\'exico\\[-1.6mm]
\small C.P. 04510, M\'exico D.F.\\[-1.6mm]
\small \texttt{sergiopalafoxd@gmail.com}
\\[2mm]
\textbf{Luis O. Silva}
\\
\small Departamento de F\'{i}sica Matem\'{a}tica\\[-1.6mm]
\small Instituto de Investigaciones en Matem\'aticas Aplicadas y en Sistemas\\[-1.6mm]
\small Universidad Nacional Aut\'onoma de M\'exico\\[-1.6mm]
\small C.P. 04510, M\'exico D.F.\\[-1.6mm]
\small \texttt{silva@iimas.unam.mx}
}
\date{}
\maketitle
\vspace{-4mm}
\begin{center}
\begin{minipage}{5in}
  \centerline{{\bf Abstract}} \bigskip This work provides a complete
  characterization of the solutions of a linear interpolation problem
  for vector polynomials. The interpolation problem consists in
  finding $n$ scalar polynomials such that an equation involving a
  linear combination of them is satisfied for each one of the $N$
  interpolation nodes. The results of this work generalize previous
  results on the so-called rational interpolation and have
  applications to direct and inverse spectral analysis of band
  matrices.
\end{minipage}
\end{center}
\thispagestyle{empty}
\end{titlepage}
\section{Introduction}
\label{sec:intro}
In this work, we are concerned with the following interpolation
problem.  Given a collection of complex numbers $z_1,\dots,z_N$,
which are called interpolation nodes,
and other collections of complex numbers
$\alpha_{k}(1),\dots,\alpha_{k}(N)$, $k=1,\dots,n$, such that
$\sum_{k=1}^{n}\left|\alpha_{k}(j)\right|>0$ for every
$j\in\left\{1,2,\ldots,N\right\}$, find polynomials $P_{k}$,
$k=1,\ldots,n$, which satisfy
\begin{align}
  \sum_{k=1}^{n}\alpha_k(j)P_{k}(z_j)=0\qquad\forall
  j\in\left\{1,2,\ldots,N\right\}\,.\label{eq:linear-combination}
\end{align}
We lay stress on the fact that the interpolation nodes
$z_1,z_2,\dots,z_N$ are not required to be pairwise different (see
Remark~\ref{rem:equal-nodes} in Section~\ref{sec:generators}).  The
results of this paper give a complete characterization of all solutions of the
interpolation problem (\ref{eq:linear-combination}).

The interpolation problem defined above has been studied in
\cite{MR2533388} and much earlier in \cite{MR1091797} for the
particular setting when $n=2$. In this case, the theory developed in
\cite{MR2533388,MR1091797} allows to treat the problem of finding a
rational function $P_1(z)/P_2(z)$ which takes the value
$-\alpha_2(j)/\alpha_1(j)\in\cc{\complex}$ at each interpolation node
$z_j$. This is the so-called rational interpolation problem, or
Cauchy-Jacobi problem, and $P_1(z)/P_2(z)$ is referred to as the
multipoint Pad\'e approximant
\cite[Sec.\,7.1]{MR1383091}. Noteworthily, although the research in
this matter has put particular emphasis on the numerical aspect of the
problem, \cite{MR2533388,MR1091797} consider the theoretical problem
of accounting for the structure of all solutions of the rational
interpolation problem. In \cite{MR2533388}, 
this is used to deal with the inverse spectral analysis of
five-diagonal unitary matrices (the so called CMV matrices,
cf. \cite{MR2494240}) and this requires additional constraints for
certain coefficients of the interpolating rational
polynomials. Similarly, the spectral analysis of five-diagonal
symmetric matrices also demands additional conditions (see
\cite{MR1699440}) on the rational interpolation problem.  The
description obtained in \cite{MR2533388} permits to reduce the
rational interpolation problem with such additional restrictions to a
triangular linear system and to answer the specific question of the
existence and uniqueness (or non-uniqueness) of the solution to the
inverse spectral problem. Other approaches to rational interpolation
can be found in \cite{MR1888130,MR2477608,MR2243467}.

The present work generalizes to any $n\in\nats$ the linear
interpolation theory given in \cite[Sec.\,2]{MR2533388}. The passage
from $n=2$ to any $n\in\nats$ is not straightforward; many of the
obtained results require differing techniques. Particularly, this
becomes clear in Sections~\ref{sec:generators} and
\ref{sec:characterization-solutions}. Similar generalizations of the
rational interpolation problem that also focus on the structure of the
set of solutions can be found in
\cite{MR1091770,MR1199353,MR1092122,MR1199391}.

Our main motivation for studying the interpolation problem given by
(\ref{eq:linear-combination}) lies in its applications to direct and
inverse spectral problems of $N\times N$ symmetric band matrices with
$2n+1$ diagonals, which will be considered in a forthcoming paper
\cite{2014arXiv1409.3868K}. Notwithstanding the fact that the
interpolation theory discussed in this work was developed with the
applications to inverse spectral analysis in mind, we solve a problem
interesting by itself and which may have other applications.  It is
worth remarking that, although
\cite{MR1091770,MR1199353,MR1092122,MR1199391} also deal with the
structural properties of the solution set of the interpolation problem
given by \eqref{eq:linear-combination}, our approach differs in
several respects from the ones used in those works. On the one hand,
this permits to tackle the inverse spectral analysis of finite
diagonal band matrices \cite{2014arXiv1409.3868K}. On the other hand,
the method developed here allows a new characterization of the
interpolation problem and new results on the structure of the
solutions (see Secition~\ref{sec:characterization-solutions} and, in
particular, Theorems~\ref{thm:conjecture-equality} and
\ref{thm:all-solution-comb-generators}).

The exposition is organized as follows. In Section~\ref{sec:vect-poly}
we lay down the notation, introduce the main concepts, and prove some
subsidiary assertions. Section~\ref{sec:tranformation} contains
auxiliary results related to linear transformations of vector
polynomials. Finally, in Sections~\ref{sec:generators} and
\ref{sec:characterization-solutions}, we show that the so-called
\emph{generators} determine the set of solutions of the interpolation
problem given by \eqref{eq:linear-combination} and provide a complete
characterization of this set.

\section{Vector polynomials and their height}
\label{sec:vect-poly}

Throughout this work we consider the number $n\in\mathbb{N}$ to be
fixed. We begin this section by fixing the notation and introducing
some auxiliary concepts.
\begin{definition}
  Let us denote by $\mathbb{P}$ the space of $n$-dimensional vector
  polynomials, viz.,
\begin{align*}
\mathbb{P}:=\left\{\boldsymbol{p}(z)=\left(
\begin{array}{c}
P_{1}(z)\\
P_{2}(z)\\
\vdots\\
P_{n}(z)
\end{array}
\right):\; P_{k}\;\text{is a scalar polynomial for}\;
k\in\{1,\dots,n\}\;\right\}.
\end{align*}
\end{definition}
Clearly, $\mathbb{P}$ is an infinite dimensional linear space and it
is a module over the ring of scalar polynomials, i.\,e, for any scalar
polynomial $S$,
\begin{equation*}
  \boldsymbol{p}\in\mathbb{P}\Rightarrow S\boldsymbol{p}=
\left(SP_{1}(z),SP_{2}(z),\dots,SP_{n}(z)\right)^{t}\in\mathbb{P}\,.
\end{equation*}

\begin{definition}
\label{def:height-polyn-vect}
Let the function $h:\mathbb{P}\rightarrow\mathbb{N}\cup\{0,-\infty\}$
be defined by
\begin{equation}
\label{eq:height-definition}
h(\boldsymbol{p}):=\begin{cases}
\max_{j\in\{1,\dots,n\}}\left\{n\deg
  P_{j}(z)+j-1\right\},&\boldsymbol{p}\neq0\,,
\\
-\infty,&\boldsymbol{p}=0\,,
\end{cases}
\end{equation}
where it has been assumed that $\deg 0=-\infty$.
The number $h(\boldsymbol{p})$ is called the height of the vector
polynomial $\boldsymbol{p}$.
\end{definition}

Note that for any scalar polynomial $S$
\begin{equation}
\label{eq:height-Sp}
h(S\boldsymbol{p})=h(\boldsymbol{p})+ n\deg S.
\end{equation}

\begin{lemma}
\label{lem:height-properties}
\begin{enumerate}[(a)]
\item If $h(\boldsymbol{p})\neq h(\boldsymbol{q})$, then
  $h(a\boldsymbol{p} + b\boldsymbol{q})=\max
  \left\{h(\boldsymbol{p}),h(\boldsymbol{q})\right\}$ for all
  $a,b\in\mathbb{C}$.
\label{lem:height-properties-a}
\item If $h(\boldsymbol{p})=h(\boldsymbol{q})=m$, then
  $h(a\boldsymbol{p}+b\boldsymbol{q}) \leq m$ for every $a,b\in\mathbb{C}$.
\label{lem:height-properties-b}
\item If $h(\boldsymbol{p})=h(\boldsymbol{q})=m$, then there exists a
  $c\in\mathbb{C}$ such that
$h(\boldsymbol{p}+c\boldsymbol{q})\leq m-1$
\label{lem:height-properties-c}
\end{enumerate}
\end{lemma}

\begin{proof}
  We only prove (\ref{lem:height-properties-c}) since
  (\ref{lem:height-properties-a}) and (\ref{lem:height-properties-b})
  are proven with the same argumentation. Let $m=nk+l$ with
  $l\in\{0,1,2,\ldots,n-1\}$ and $k\in\mathbb{N}\cup \{0\}$, then $k$
  and $l$ are uniquely determined by $m$ and
  $\deg(Q_{l+1}(z))=\deg(P_{l+1}(z))=k$. Therefore, there is $c$ so
  that $\deg (P_{l+1}+cQ_{l+1})\leq k-1$. Also, one has that $\deg
  (P_{j}+cQ_{j})$ is not greater than $k-1$ for $l+1\leq j\leq n$ and
  $\deg (P_{j}+cQ_{j})$ is not greater than $k$ for $1\leq j\leq
  l$. So,
\begin{align*}
h(\boldsymbol{p}+c\boldsymbol{q})=& \max_{j\in\{1,\dots,n\}}
\left\{n\:\deg(P_{j}(z)+cQ_{1}(z))
  + j-1\right\}\\
\leq & \max\left\{nk+l-1,n(k-1)+l,n(k-1)+n-1\right\}\\
\leq & nk+l-1=m-1.
\end{align*}
\end{proof}

For $k=0,1,\dots,$ let us consider the following set of elements in
$\mathbb{P}$,
\begin{align}
\label{eq:basis}
\boldsymbol{e}_{nk+1}(z)=\left(\begin{array}{c}
z^{k}\\
0\\
0\\
\vdots \\
0
\end{array}\right),\ \boldsymbol{e}_{nk+2}(z)=\left(\begin{array}{c}
0\\
z^k\\
0\\
\vdots \\
0
\end{array}\right),\ \dots\ , \ \boldsymbol{e}_{n(k+1)}(z)=\left(\begin{array}{c}
0\\
0\\
0\\
\vdots \\
z^k
\end{array}\right).
\end{align}
Clearly, $h(\boldsymbol{e}_j(z))=j-1$ for all $j\in\mathbb{N}$.

\begin{lemma}
\label{lem:basis-height}
The sequence $\left\{\boldsymbol{e}_j(z)\right\}_{j=1}^{\infty}$ is a
basis of the space $\mathbb{P}$, i.\,e., for any
$\boldsymbol{p}\in\mathbb{P}$ with $h(\boldsymbol{p})=m\ne-\infty$, there exist unique
numbers $c_0,c_1,\ldots,c_m$, where $c_m\neq0$, such that
\begin{align*}
\boldsymbol{p}(z)=\sum_{k=0}^m c_k\boldsymbol{e}_{k+1}.
\end{align*}
\end{lemma}

\begin{proof}
  We prove the assertion by induction.  If $m=0,1,\ldots,n-1$, the
  result is immediate.  As before, let $m=nk+l$ with $k\in\mathbb{N}$
  and $l\in\{0,1,\dots,n-1\}$. Then one can write
  $P_{l+1}(z)=az^k+Q_{l+1}(z)$, where $a\neq0$ and $\deg Q_{l+1}\leq
  k-1$.

  Define $\boldsymbol{q}(z):= \boldsymbol{p}(z) -
  a\boldsymbol{e}_{nk+l+1}= \left(Q_{1}(z),Q_{2}(z), \dots,
    Q_{n}(z)\right)^{t}$, i.e. $Q_{j}=P_{j}$ for
  $j=\{1,2,\dots,n\}\setminus\{l+1\}$. Thus, $\deg Q_{j}$ is not
  greater than $k-1$ for all $j=l+2,l+3,\ldots,n$ and $\deg Q_{j}$ is
  not greater than $k$ for every $j=1,2,\ldots,l$. Therefore, one has
\begin{align*}
h(\boldsymbol{q})\leq \max\left\{nk+l-1,n(k-1)+n-1\right\}=m-1\,.
\end{align*}
In the induction hypothesis we assume
$\boldsymbol{q}=\sum_{k=0}^{m-1}c_k\boldsymbol{e}_{k+1}$.  So, one
obtains
\begin{align*}
  \boldsymbol{p}=a\boldsymbol{e}_{nk+l+1}+\boldsymbol{q}=
  a\boldsymbol{e}_{m+1} + \sum_{k=0}^{m-1}c_k\boldsymbol{e}_{k+1}=
\sum_{k=0}^{m}\widetilde{c_k}\boldsymbol{e}_{k+1}.
\end{align*}
The uniqueness of the expansion follows from the linear independence
of the sequence $\left\{\boldsymbol{e}_j(z)\right\}_{j=1}^{\infty}$,
which is straightforward to verify.
\end{proof}

\begin{theorem}
\label{thm:basis-height}
Let $\left\{\boldsymbol{g}_m\right\}_{m=1}^{\infty}$ be an arbitrary
sequence of elements in $\mathbb{P}$ such that
\begin{align*}
h(\boldsymbol{g}_m)=m-1 \qquad \forall m\in\mathbb{N},
\end{align*}
then $\left\{\boldsymbol{g}_m\right\}_{m=1}^{\infty}$ is a basis of
$\mathbb{P}$.
\end{theorem}
\begin{proof}
  From Lemma~\ref{lem:basis-height}, it follows that
  $\boldsymbol{g}_{m+1}(z)=\sum_{k=0}^{m}c_{mk}\boldsymbol{e}_{k+1}(z)$,
  where $c_{jj}$ is different from $0$ for all $j=0,\ldots,m$. So
\begin{align*}
\left(\begin{array}{c}
\boldsymbol{g}_1\\
\boldsymbol{g}_2\\
\vdots\\
\boldsymbol{g}_{m+1}
\end{array}\right)=\left(\begin{array}{cccc}
c_{00}&0&\ldots&0\\
c_{10}&c_{11}&\ldots&0\\
\vdots\\
c_{m0}&c_{m1}&\ldots&c_{mm}
\end{array}\right)\left(\begin{array}{c}
\boldsymbol{e}_1\\
\boldsymbol{e}_2\\
\vdots\\
\boldsymbol{e}_{m+1}
\end{array}\right).
\end{align*}
Note that $\{c_{jk}\}_{j,k\in\{0,\dots,m\}}$ is a triangular matrix, thus
\begin{align*}
\left(\begin{array}{c}
\boldsymbol{e}_1\\
\boldsymbol{e}_2\\
\vdots\\
\boldsymbol{e}_{m+1}
\end{array}\right)=\left(\begin{array}{cccc}
\widetilde{c}_{00}&0&\ldots&0\\
\widetilde{c}_{10}&\widetilde{c}_{11}&\ldots&0\\
\vdots\\
\widetilde{c}_{m0}&\widetilde{c}_{m1}&\ldots&\widetilde{c}_{mm}
\end{array}\right)\left(\begin{array}{c}
\boldsymbol{g}_1\\
\boldsymbol{g}_2\\
\vdots\\
\boldsymbol{g}_{m+1}
\end{array}\right).
\end{align*}
Since $\left\{\boldsymbol{e}_m\right\}_{m\in\mathbb{N}}$ is a basis,
the same is true for $\left\{\boldsymbol{g}_m\right\}_{m\in\mathbb{N}}$.
\end{proof}

\section{The height under linear transformations on $\mathbb{P}$}
\label{sec:tranformation}

Let $A=\{a_{jk}\}_{j,k\in\{1,\dots,n\}}$ be an arbitrary $n\times n$
matrix. For any $\boldsymbol{p}\in\mathbb{P}$, the linear
transformation generated by $A$ is

\begin{align*}
A\boldsymbol{p}(z)=
\left(\begin{array}{cccc}
a_{11}P_{1}(z)+a_{12}P_{2}(z)+\cdots +a_{1n}P_{n}(z)\\
a_{21}P_{1}(z)+a_{22}P_{2}(z)+\cdots +a_{2n}P_{n}(z)\\
\vdots\\
a_{n1}P_{1}(z)+a_{n2}P_{2}(z)+\cdots +a_{nn}P_{n}(z)
\end{array}\right)\in\mathbb{P}.
\end{align*}

It follows from Definition~\ref{def:height-polyn-vect} that
\begin{equation}
\label{eq:height-Ap}
h(A\boldsymbol{p})=
\max_{j\in\{1,\dots,n\}}\left\{n\deg\left(\sum_{k=1}^{n}a_{jk}P_{k}(z)\right)
+j-1\right\}.
\end{equation}

\begin{lemma}
\label{lem:height-transf-matrix}
\begin{enumerate}[(a)]
\item For any arbitrary $n\times n$ matrix $A$ and $\boldsymbol{p}\in\mathbb{P}$,
\begin{equation*}
h(A\boldsymbol{p})\leq h(\boldsymbol{p})+n-1.
\end{equation*}\label{item:height-transf-matrix-a}
\vspace{-6mm}
\item If the $n\times n$ matrix $A$ is upper triangular then for any $\boldsymbol{p}\in\mathbb{P}$,
\begin{equation*}
h(A\boldsymbol{p})\leq h(\boldsymbol{p}).
\end{equation*}\label{item:height-transf-matrix-upper}
\vspace{-6mm}
\item If the $n\times n$ matrix $A$ is lower triangular, then $h(\boldsymbol{p})\leq nk+n-1$
  implies $h(A\boldsymbol{p})\leq nk+n-1$ for any
  $k\in\mathbb{N}\cup\{0\}$.\label{item:height-transf-matrix-low}
\end{enumerate}
\end{lemma}
\begin{proof}
  (\ref{item:height-transf-matrix-a}) Note that, for all
  $j\in\{1,\dots n\}$ the inequalities below hold
\begin{align*}
  n\deg\left(\sum_{k=1}^{n}a_{jk}P_{k}\right)+j-1\leq
  n\max_{k\in\{1,\dots,n\}}\left\{\deg P_{k}\right\}+j-1\leq
  h(\boldsymbol{p})+j-1.
\end{align*}
Hence $h(A\boldsymbol{p})\leq h(\boldsymbol{p})+n-1$.

(\ref{item:height-transf-matrix-upper}) Let the matrix
$A=\left\{a_{jk}\right\}_{j,k\in\{1,\dots,n\}}$ be such that $a_{jk}=0$ if
$j>k$. Then, for the last entry of the vector polynomial
$A\boldsymbol{p}$, one has
\begin{align*}
  n\deg(a_{nn}P_{n})+n-1&\leq h(\boldsymbol{p}),
\end{align*}
and for the next to last
\begin{align*}
  n\deg\left(\sum_{k=n-1}^{n}a_{n-1,k}P_{k}\right)+n-2\leq
  n\max_{k\in\{n-1,n\}}\left\{\deg P_{k}\right\}+n-2\leq h(\boldsymbol{p}).
\end{align*}
Analogously, one obtains inequalities for all the entries up to the
first one:
\begin{align*}
  n\deg(\sum_{j=1}^{n}a_{1k}P_{k})\leq
  n\max_{k\in\{1,\dots,n\}}\left\{\deg P_{k}\right\}\leq h(\boldsymbol{p}).
\end{align*}
Therefore, $h(A\boldsymbol{p})\leq h(\boldsymbol{p})$.

(\ref{item:height-transf-matrix-low}) Let
$A=\left\{a_{jk}\right\}_{j,k\in\{1,\dots,n\}}$ be such that $a_{jk}=0$ if
$j<k$. And $h(\boldsymbol{p})\leq nl+n-1$ with $l\in\mathbb{N}\cup\{0\}$. One verifies that
\begin{align*}
  n\deg P_{j}+j-1\leq nl+n-1\quad \forall j=1,\ldots,n\,,
\end{align*}
therefore $\deg P_{j}\leq l+\frac{n-j}{n}$ for all
$j=1,\ldots,n$. This implies that $\deg P_{j}\leq l$ for any
$j=1,\ldots,n$ . So by (\ref{eq:height-Ap}), $h(A\boldsymbol{p})\leq
\max_{j\in\{1,\dots,n\}}\{nl+j-1\}=nl+n-1$.
\end{proof}
Now, we introduce some matrices and state auxiliary results for
them. These results will be useful in the next section.

Let $A_l=\left\{a_{jk}\right\}_{j,k\in\left\{1,\ldots n\right\}}$ be
such that, for a fixed integer $l\in\{1,\dots,n\}$, it satisfies
\begin{align*}
  a_{jj}=1\qquad& \forall j\in\left\{1,\ldots, n \right\}\setminus
  \left\{l\right\}\,,\\
  a_{jk}=0\qquad& \forall j\neq k,\; {\rm with} \; j\in\left\{1,\ldots,
    n\right\}\setminus\left\{l\right\}\,,
\end{align*}
that is, it has the form

\begin{align}\label{A_l}
A_l=\left(
\begin{array}{cccccccc}
1&0&\ldots&0&0&0&\ldots&0\\
0&1&\ldots&0&0&0&\ldots&0\\
\vdots&\vdots&\ddots&\vdots&\vdots&\vdots&\ddots&\vdots\\
0&0&\ldots&1&0&0&\ldots&0\\
a_{l1}&a_{l2}&\ldots&a_{l\:l-1}&a_{ll}&a_{l\:l+1}&\ldots&a_{ln}\\
0&0&\ldots&0&0&1&\ldots&0\\
\vdots&\vdots&\ddots&\vdots&\vdots&\vdots&\ddots&\vdots\\
0&0&\ldots&0&0&0&\ldots&1\\
\end{array}
\right).
\end{align}

Also, for any $l\in\{ 1,\dots,n \}$, define the matrix function
\begin{equation}
\label{eq:matriz-t_j}
T_l(z):=\diag\{t_k(z)\}\,,\quad
t_k(z):=\
\begin{cases}
  z & \text{if } k=l\,,\\
  1 & \text{otherwise}\,,
\end{cases}
\end{equation}
i.\,e., $T_l(z)$ is nearly the identity matrix, except that in the
$l$-th entry of the main diagonal, $T_l(z)$ has the variable $z$
instead of $1$.
\begin{lemma}
\label{lem:height-transf-particular-matrix}
Fix $n\geq 3$ and $l\in\left\{2,\ldots,n-1\right\}$. If
$\boldsymbol{p}\in\mathbb{P}$ is such that $h(\boldsymbol{p})\leq
nk+l-1$ for any $k\in\mathbb{N}\cup\left\{0\right\}$, then
$h(A_l\boldsymbol{p})\leq nk+l-1$.
\end{lemma}
\begin{proof}
Let $\boldsymbol{p}\in\mathbb{P}$. If $h(\boldsymbol{p})\leq nk+l-1$, then
\begin{align}
\deg P_l\geq \deg P_i\,,\quad 1\leq i\leq l-1\,, \label{eq:particular-matrix-<m}\\
\deg P_l>\deg P_i\,,\quad l+1\leq i\leq n\,. \label{eq:particular-matrix->m}
\end{align}
On the other hand,
\begin{equation*}
A_l\boldsymbol{p}=\left(P_1(z),P_2(z),\dots,P_{l-1}(z),
  \sum_{i=1}^{n}a_{li}P_i(z),P_{l+1}(z),\dots,P_n(z)\right)^{t}\,.
\end{equation*}
By (\ref{eq:particular-matrix-<m}) and
(\ref{eq:particular-matrix->m}) we have
\begin{equation*}
  \deg\left(\sum_{i=1}^{n}a_{li}P_i(z)\right)\leq\max_{i\in\left\{1,\ldots,
    n\right\}}\left\{\deg P_i(z)\right\}\leq \deg P_l(z)\,.
\end{equation*}
Hence, by (\ref{eq:height-Ap}),
\begin{equation*}
h(A_l\boldsymbol{p})\leq h(\boldsymbol{p})\leq nk+l-1.
\end{equation*}
\end{proof}

\begin{lemma}
\label{lem:t_j-modified}
If $\boldsymbol{p}\in\mathbb{P}$ and $h(\boldsymbol{p})$ is not
greater than $nk+j$ for
$k\in\mathbb{N}\cup\{0\}$ and $j\in\left\{0,1,\ldots,n-1\right\}$, then
$h(T_{j+2}(z)\boldsymbol{p})\leq nk+j+1$, where $T_{n+1}:=T_1$.
\end{lemma}
\begin{proof}
  The assertion follows from (\ref{eq:particular-matrix-<m}) and
  (\ref{eq:particular-matrix->m}) by a reasoning similar to the one
  used in the proof of Lemma
  \ref{lem:height-transf-particular-matrix}.
\end{proof}

\section{Generators of the interpolation problem}
\label{sec:generators}

In this section we begin the detailed analysis of the interpolation
problem set forth in the Introduction. Let us first provide an
alternative interpretation of the interpolation problem given by
(\ref{eq:linear-combination}).

Clearly, for all
   $j\in\{1,\dots,N\}$, one has
\begin{equation*}
\left|\sum_{k=1}^n\alpha_{k}(j)P_k(z_j)\right|^2\!=
\left(\sum_{k=1}^n\overline{\alpha_{k}(j)P_k(z_j)}\right)
\left(\sum_{k=1}^n\alpha_{k}(j)P_k(z_j)\right)=
\inner{\boldsymbol{p}(z_j)}{\sigma_j\boldsymbol{p}(z_j)}\,,
\end{equation*}
where $\inner{\cdot}{\cdot}$ is the inner product in $\complex^n$ with
the first argument being anti-linear, and
\begin{align}\sigma_j:=\left(
\label{eq:sigma-matrix}
 \begin{array}{ccccc}
 \left|\alpha_1(j)\right|^2&\overline{\alpha_1(j)}\alpha_2(j)&\overline{\alpha_1(j)}\alpha_3(j)&\ldots&\overline{\alpha_1(j)}\alpha_n(j)\\
 \overline{\alpha_2(j)}\alpha_1(j)&\left|\alpha_2(j)\right|^2&\overline{\alpha_2(j)}\alpha_3(j)&\ldots&\overline{\alpha_2(j)}\alpha_n(j)\\
 \overline{\alpha_3(j)}\alpha_1(j)&\overline{\alpha_3(j)}\alpha_2(j)&\left|\alpha_3(j)\right|^2&\ldots&\overline{\alpha_3(j)}\alpha_n(j)\\
 \vdots&\vdots&\vdots&\ddots&\vdots\\
 \overline{\alpha_n(j)}\alpha_1(j)&\overline{\alpha_n(j)}\alpha_2(j)&\overline{\alpha_n(j)}\alpha_3(j)&\ldots&\left|\alpha_n(j)\right|^2
 \end{array}\right)\,.
 \end{align}
 Thus,
 the interpolation problem (\ref{eq:linear-combination}) is equivalent
 to finding $\boldsymbol{p}\in\mathbb{P}$ such that
 \begin{equation}
\label{eq:auxilliary-problem}
\inner{\boldsymbol{p}(z_j)}{\sigma_j\boldsymbol{p}(z_j)}=0\,,\quad j=1,\ldots,N.
 \end{equation}
 \begin{remark}
   \label{rem:alternative-interpolation}
   It follows  from (\ref{eq:sigma-matrix}) that $\sigma_j$ is a
   nonnegative rank-one matrix. Moreover, for every nonnegative rank-one
   $n\times n$ matrix $\sigma$, there is a collection of complex numbers
   $\alpha_1,\dots,\alpha_n$ such that
\begin{align}\sigma=\left(
\label{eq:sigma-matrix-generic}
 \begin{array}{ccccc}
 \left|\alpha_1\right|^2&\overline{\alpha_1}\alpha_2&\overline{\alpha_1}\alpha_3&\ldots&\overline{\alpha_1}\alpha_n\\
 \overline{\alpha_2}\alpha_1&\left|\alpha_2\right|^2&\overline{\alpha_2}\alpha_3&\ldots&\overline{\alpha_2}\alpha_n\\
 \overline{\alpha_3}\alpha_1&\overline{\alpha_3}\alpha_2&\left|\alpha_3\right|^2&\ldots&\overline{\alpha_3}\alpha_n\\
 \vdots&\vdots&\vdots&\ddots&\vdots\\
 \overline{\alpha_n}\alpha_1&\overline{\alpha_n}\alpha_2&\overline{\alpha_n}\alpha_3&\ldots&\left|\alpha_n\right|^2
 \end{array}\right)\,.
 \end{align}
Thus, the interpolation problem can be stated as the problem of
finding $\boldsymbol{p}\in\mathbb{P}$ such that
(\ref{eq:auxilliary-problem}) holds for any collection of nonnegative rank-one
 $n\times n$  matrices $\{\sigma_j\}_{j=1}^N$.
\end{remark}
\begin{definition}
\label{def:solution-problem}
Let us denote by
$\mathbb{S}(n,N)=\mathbb{S}(\{\sigma_j\}_{j=1}^N,\{z_j\}_{j=1}^N)$ the
set of all solutions of the interpolation problem
(\ref{eq:linear-combination}), where $\sigma_j$ is given by
(\ref{eq:sigma-matrix}). We use the notation $\mathbb{S}(n,N)$ when
the concrete matrices $\{\sigma_j\}_{j=1}^N$ and the interpolation
nodes $\{z_j\}_{j=1}^N$ are not relevant.
\end{definition}
Note that an interpolation problem is completely determined by the
sets $\{\sigma_j\}_{j=1}^N$ and $\{z_j\}_{j=1}^N$.  Since a solution
of (\ref{eq:linear-combination}) is an element of $\mathbb{P}$, one
obviously has $\mathbb{S}(n,N)\subset\mathbb{P}$. Clearly, in the same way it
happens for $\mathbb{P}$, the space $\mathbb{S}(n,N)$ is a module over
the ring of scalar polynomials.
\begin{remark}
  \label{rem:equal-nodes}
  Consider the interpolation problem given by $\{\sigma_j\}_{j=1}^N$ and
  $\{z_j\}_{j=1}^N$, if it turns out that $z_N=z_{N-1}$ and the
  vectors
  $\boldsymbol{\alpha}(j_N):=(\alpha_1(j_N),\dots,\alpha_n(j_N))^t$ and
  $\boldsymbol{\alpha}(j_{N-1})$ are linearly dependent, then
  \begin{equation*}
    \mathbb{S}(\{\sigma_j\}_{j=1}^N,\{z_j\}_{j=1}^N)=\mathbb{S}(\{\sigma_j\}_{j=1}^{N-1},\{z_j\}_{j=1}^{N-1})\,.
  \end{equation*}
  For the sake of convenience, we will suppose below
  that the vectors $\boldsymbol{\alpha}$ for the coinciding nodes are
  linearly independent, and in the case they are linearly dependent,
  the statements of the results should be changed in an evident manner
  to the corresponding statements with less nodes.
\end{remark}

Let $\mathbb{M}(\boldsymbol{r})$ be the subset of $\mathbb{P}$ given
by
\begin{equation}
\label{eq:generator-set}
\mathbb{M}(\boldsymbol{r}) := \left\{\boldsymbol{p}\in\mathbb{P}:
  \:\boldsymbol{p}=S\boldsymbol{r}, \: \boldsymbol{r}\in\mathbb{P}, \:
  S \:\text{ is an arbitrary scalar polynomial}\right\}.
\end{equation}
We say that $\mathbb{M}(\boldsymbol{r})$ is the set of vector
polynomials generated by $\boldsymbol{r}$. Note that
$\mathbb{M}(\boldsymbol{r})$ is a linear set and for any
nonzero $\boldsymbol{q}\in\mathbb{M}(\boldsymbol{r})$ there exists
$k\in\nats\cup\{0\}$ such that
\begin{equation}
  \label{eq:M-height}
  h(\boldsymbol{q})= h(\boldsymbol{r})+ nk\,.
\end{equation}
Thus, all nonzero vector polynomials of $\mathbb{M}(\boldsymbol{r})$
are such that their heights are in the same equivalence class of
$\integers/n\integers$.

One of the main goals of this section is to show that
$\mathbb{S}(n,N)$ has exactly $n$ generators, that
is, there are $n$ vector polynomials
$\boldsymbol{r}_1,\dots,\boldsymbol{r}_{n}$ such that
\begin{equation*}
 \mathbb{M}(\boldsymbol{r}_1)\dotplus\dots\dotplus
  \mathbb{M}(\boldsymbol{r}_{n})=\mathbb{S}(n,N)\,.
\end{equation*}
This result is related to a known fact about the set of polynomials
being solutions of the rational interpolation
problem in its vector case \cite[Thms.\,3.1,\,3.2]{MR1092122} or the M-Pad\'e
approximation problem \cite[Thm.\,3.1]{MR1199391}.
\begin{definition}
  \label{def:height-set}
  Let $\mathcal{M}$ be an arbitrary subset of
  $\mathbb{P}$. We define the height of
  $\mathcal{M}$ by
\begin{equation}
\label{eq:set-height}
h(\mathcal{M}):=\min \left\{h(\boldsymbol{q}): \:
  \boldsymbol{q}\in\mathcal{M}, \:\boldsymbol{q}\neq 0\right\}\,.
\end{equation}
\end{definition}
\begin{lemma}
\label{lem:same-height}
Let $\mathcal{M}$ be a linear subset of
$\mathbb{S}(n,N)$ and
$\boldsymbol{r},\boldsymbol{p}\in\mathcal{M}$ such that
$h(\boldsymbol{r})=h(\boldsymbol{p})=h(\mathcal{M})$, then
$\boldsymbol{r}=c\boldsymbol{p}$ with $c\in\mathbb{C}$.
\end{lemma}
\begin{proof}
  From Lemma
  \ref{lem:height-properties}(\ref{lem:height-properties-c}), it
  follows that there exists a complex constant $c$ in $\mathbb{C}$
  such that $h(\boldsymbol{r}+c\boldsymbol{p})\leq
  h(\mathcal{M})-1$. Since $\mathcal{M}$ is linear,
  $\boldsymbol{r}+c\boldsymbol{p}\in\mathcal{M}$, but there is no
  element $\boldsymbol{q}\not\equiv 0$ in $\mathcal{M}$ such that
  $h(\boldsymbol{q})\leq h(\mathcal{M})$. Hence
  $\boldsymbol{r}+c\boldsymbol{p}\equiv0$.
\end{proof}
\begin{definition}
\label{def:0-generator}
We say that $\boldsymbol{r}$ in $\mathbb{S}(n,N)$ is a first generator of
$\mathbb{S}(n,N)$ when $h(\boldsymbol{r})=h(\mathbb{S}(n,N))$.
\end{definition}
Let us denote by $\mathbb{M}_1$ the set $\mathbb{M}(\boldsymbol{r})$
with $\boldsymbol{r}$ being a first generator. Clearly, Lemma
\ref{lem:same-height} implies that $\mathbb{M}_1$ does not depend on
the choice of the first generator.
\begin{theorem}
\label{thm:height-bound-S_n}
If $\boldsymbol{r}$ is a first generator of $\mathbb{S}(n,N)$, then
$h(\boldsymbol{r})\leq N$ for any $N\in\mathbb{N}$.
\end{theorem}
\begin{proof}
  The goal of this proof is to obtain a constructive algorithm for
  finding a solution whose height is not greater than $N$. Clearly,
  this implies the assertion of the theorem since, by definition, the
  height of the first generator is less than or equal to the height of
  any nonzero solution.

  Our construction is carried out by induction.  For $N=1$, we have a
  solution
 \begin{equation}
\label{eq:inequalities-nodes-n1}
 \boldsymbol{p}(z):=\left(C_1,C_2,0,\dots, 0\right)^{t}\,,
 \end{equation}
 where $C_1=\alpha_2(1)$, $C_2=-\alpha_1(1)$, unless $\alpha_2(1)$ and
 $\alpha_1(1)$ are both zero, in which case $C_1,C_2$ are any nonzero
 constants. Indeed, (\ref{eq:inequalities-nodes-n1}) is solution of
 (\ref{eq:auxilliary-problem}) since
 $\inner{\boldsymbol{p}(z_1)}{\sigma_1\boldsymbol{p}(z_1)}=0$, and
 $h(\boldsymbol{p})\le 1$.

 Now, we suppose that the assertion holds for a fixed $N$ and let us
 show that it also holds for $N+1$. We will reduce the interpolation
 problem with $N+1$ nodes to an interpolation problem with $N$
 nodes, which we know how to solve by the induction hypothesis.

 Write $N=nk+l$ with $l<n$, where, for any fixed $n$, the integers $k$
 and $l$ are uniquely determined. Suppose that $n\geq 3$ and let us
 first prove the assertion for a fixed $l$ in $\{0,1,\dots,n-3\}$.  If
 the matrix (\ref{eq:sigma-matrix}) satisfies that $\alpha_{l+2}(j)=0$
 for all $j\in\left\{1,\ldots,N+1\right\}$, then
 $\sigma_j=T_{l+2}(0)\sigma_j T_{l+2}(0)$. So, by putting the vector polynomial
 $\boldsymbol{p}(z):=\boldsymbol{e}_{l+2}(z)$ (see
 (\ref{eq:basis})), it turns out that $\boldsymbol{p}$ is
 an element of $\mathbb{S}(n,N+1)$. Indeed, since
 $T_{l+2}(0)\boldsymbol{p}(z)=0$, one has
\begin{align*}
\inner{\boldsymbol{p}(z_j)}{\sigma_j\boldsymbol{p}(z_j)} =
\inner{\boldsymbol{p}(z_j)}{T_{l+2}(0)\sigma_j
  T_{l+2}(0)\boldsymbol{p}(z_j)}
 = 0.
\end{align*}
Moreover, $h(\boldsymbol{p})=l+1\leq N+1$. Thus, suppose without loss of
generality that $\alpha_{l+2}(N+1)$ is not equal to zero (otherwise
re-enumerate the points $z_1,\dots,z_{N+1}$).

Consider the matrix $A_{l+2}$ (see (\ref{A_l})), where
\begin{equation*}
a_{l+2,k}:=
\begin{cases}
  \frac{1}{\alpha_{l+2}(N+1)} & \text{if } k=l+2\,,\\[3mm]
-\frac{\alpha_k(N+1)}{\alpha_{l+2}(N+1)} & \text{for } k\in\left\{1,2,\ldots,n\right\}\setminus\left\{l+2\right\}\,.
\end{cases}
\end{equation*}
With these settings, it is straightforward to verify that
\begin{equation}
\label{eq:A-sigma-A-n+1}
A_{l+2}^*\sigma_{N+1}A_{l+2}=
\diag\{d_k\}_{k=1}^n\,,\quad
d_k:=\
\begin{cases}
  1 & \text{if } k=l+2\,,\\
  0 & \text{otherwise}\,,
\end{cases}
\end{equation}
Moreover, since $\sigma_j$ is a nonnegative rank-one matrix for every
$j\in\{1,\dots,N\}$, the same is true for
$A_{l+2}^*\sigma_{j}A_{l+2}$. Therefore (see Remark~\ref{rem:alternative-interpolation}),
for any $j\in\left\{1,\ldots,N\right\}$, there are complex numbers
$\beta_1(j),\ldots,\beta_n(j)$ such that
\begin{align}
\label{eq:A-sigma-A-j}
A_{l+2}^*\sigma_{j}A_{l+2}=\left(
 \begin{array}{ccccc}
 \left|\beta_1(j)\right|^2&\overline{\beta_1(j)}\beta_2(j)&\overline{\beta_1(j)}\beta_3(j)&\ldots&\overline{\beta_1(j)}\beta_n(j)\\
 \overline{\beta_2(j)}\beta_1(j)&\left|\beta_2(j)\right|^2&\overline{\beta_2(j)}\beta_3(j)&\ldots&\overline{\beta_2(j)}\beta_n(j)\\
 \overline{\beta_3(j)}\beta_1(j)&\overline{\beta_3(j)}\beta_2(j)&\left|\beta_3(j)\right|^2&\ldots&\overline{\beta_3(j)}\beta_n(j)\\
 \vdots&\vdots&\vdots&\ddots&\vdots\\
 \overline{\beta_n(j)}\beta_1(j)&\overline{\beta_n(j)}\beta_1(j)&\overline{\beta_n(j)}\beta_3(j)&\ldots&\left|\beta_n(j)\right|^2
 \end{array}\right)\,.
\end{align}

Now, for all $j\in\left\{1,\ldots,N\right\}$, let
\begin{equation*}
  \gamma_i(j):=
  \begin{cases}
    (z_{N+1}-z_j)\beta_{l+2}(j) & \text{if } i=l+2\,,\\
    \beta_i(j) & \text{for } i\in\left\{1,2,\ldots,n\right\}\setminus\left\{l+2\right\}\,,
  \end{cases}
\end{equation*}
and consider the auxiliary interpolation problem given by
$\{\widetilde{\sigma}_j\}_{j=1}^N$ and $\{z_j\}_{j=1}^N$, where
 \begin{equation*}
 \widetilde{\sigma_j}:=\left(
 \begin{array}{ccccc}
   \left|\gamma_1(j)\right|^2&\overline{\gamma_1(j)}\gamma_2(j)&\overline{\gamma_1(j)}\gamma_3(j)&\ldots&\overline{\gamma_1(j)}\gamma_n(j)\\
   \overline{\gamma_2(j)}\gamma_1(j)&\left|\gamma_2(j)\right|^2&\overline{\gamma_2(j)}\gamma_3(j)&\ldots&\overline{\gamma_2(j)}\gamma_n(j)\\
   \overline{\gamma_3(j)}\gamma_1(j)&\overline{\gamma_3(j)}\gamma_2(j)&\left|\gamma_3(j)\right|^2&\ldots&\overline{\gamma_3(j)}\gamma_n(j)\\
   \vdots&\vdots&\vdots&\ddots&\vdots\\
   \overline{\gamma_n(j)}\gamma_1(j)&\overline{\gamma_n(j)}\gamma_2(j)&\overline{\gamma_n(j)}\gamma_3(j)&\ldots&\left|\gamma_n(j)\right|^2
 \end{array}\right)\,.
 \end{equation*}
 By the induction hypothesis, there is a vector polynomial
 $\boldsymbol{q}$ in
 $\mathbb{S}(\{\widetilde{\sigma}_j\}_{j=1}^N$, $\{z_j\}_{j=1}^N)$ such that
\begin{equation}
  \label{eq:induction-hypothesis}
  h(\boldsymbol{q})\leq N=nk+l\,.
\end{equation}
Define the vector polynomial
 \begin{align}
\label{eq:rz:=-mt_n-m+1z_n+1}
 \boldsymbol{r}(z):=A_{l+2}T_{l+2}(z_{N+1}-z)\boldsymbol{q}(z).
\end{align}
Then, for all $j\in\left\{1,\ldots,N\right\}$,
\begin{align}
\inner{\boldsymbol{r}(z_j)}{\sigma_j\boldsymbol{r}(z_j)} =
&\inner{\boldsymbol{q}(z_j)}{T_{l+2}\left(\overline{z_{N+1}-z_j}\right)
A_{l+2}^*\sigma_jA_{l+2}T_{l+2}(z_{N+1}-z_j)\boldsymbol{q}(z_j)}\nonumber\\
=&\inner{\boldsymbol{q}(z_j)}{\widetilde{\sigma_j}\boldsymbol{q}_{z_j}}=0\,.
\label{eq:part-solution-interpolation}
\end{align}
Also, it follows from
(\ref{eq:A-sigma-A-n+1}) that
\begin{equation*}
  T_{l+2}^*(0)A_{l+2}^*\sigma_{N+1}A_{l+2}T_{l+2}(0)=0\,.
\end{equation*}
Hence
\begin{align*}
\inner{\boldsymbol{r}(z_{N+1})}{\sigma_{N+1}\boldsymbol{r}(z_{N+1})} =&
\inner{\boldsymbol{q}(z_{N+1})}{T_{l+2}(0)A_{l+2}^*\sigma_{N+1}
A_{l+2}T_{l+2}(0) \boldsymbol{q}(z_{N+1})}\\
=&0.
\end{align*}
This last equality and (\ref{eq:part-solution-interpolation}) imply
that $\boldsymbol{r}$ is in
$\mathbb{S}(\{\sigma_j\}_{j=1}^{N+1},\{z_j\}_{j=1}^{N+1})$. Moreover,
it follows from (\ref{eq:induction-hypothesis}) and (\ref{eq:rz:=-mt_n-m+1z_n+1}), by means of Lemmas
\ref{lem:height-transf-particular-matrix} and \ref{lem:t_j-modified},
that
\begin{equation*}
h(\boldsymbol{r})\leq nk+l+1=N+1.
\end{equation*}
Thus, the assertion of the theorem has been proven for $n\ge 3$ and
$l\in\{0,\dots,n-3\}$.

For proving the assertion when $l=n-2$, consider
\begin{equation*}
  \boldsymbol{r}(z):=A_nT_n(z_{N+1}-z)\boldsymbol{q}(z)
\end{equation*}
and, repeating the reasoning above, it is shown that $\boldsymbol{r}$
is in $\mathbb{S}(n,N+1)$. Moreover, since $h(\boldsymbol{q})<nk+n-2$,
Lemma~\ref{lem:t_j-modified} implies that
\begin{equation*}
  h(T_n(z_{N+1}-z)\boldsymbol{q})\le nk+n-1\,.
\end{equation*}
Therefore, Lemma
\ref{lem:height-transf-matrix}(\ref{item:height-transf-matrix-low})
yields $h(\boldsymbol{r})\le N+1$.

The case $l=n-1$ is treated analogously with
\begin{equation*}
  \boldsymbol{r}(z):=A_1T_1(z_{N+1}-z)\boldsymbol{q}(z)
\end{equation*}
being an element of $\mathbb{S}(n,N+1)$. Again, by
Lemma~\ref{lem:t_j-modified},
\begin{equation*}
  h(T_1(z_{N+1}-z)\boldsymbol{q})\le nk+n-1\,.
\end{equation*}
Thus, it follows from
Lemma~\ref{lem:height-transf-particular-matrix}(b) that
$h(\boldsymbol{r})\le N+1$.

It is now clear how to finish the proof when $n<3$.
\end{proof}

\begin{lemma}
\label{lem:existence-solution}
Given an integer $m\geq Nn$, there exists a solution $\boldsymbol{p}$ of
$\mathbb{S}(n,N)$ such that $h(\boldsymbol{p})=m$.
\end{lemma}
\begin{proof}
  Let $m=(N+k)n+l$ with $k\in\mathbb{N}\cup\{0\}$ and
  $l\in\left\{0,1,\dots,n-1\right\}$. Let us construct a solution
  $\boldsymbol{p}\in\mathbb{S}(n,N)$ such that $h(\boldsymbol{p})=(N+k)n+l$.

Define $\boldsymbol{p}$ as follows
\begin{equation*}
\boldsymbol{p}(z):=\left(0,\dots,0,P_{l+1}(z),0,\dots,0\right)^{t}
\end{equation*}
where $P_{l+1}(z)=z^k\prod_{j=1}^{N}(z-z_j)$ and $z_1,\dots,z_N$ are
the nodes of the interpolation problem
(\ref{eq:linear-combination}). It is straightforward to verify that
$p$ is solution of $\mathbb{S}(n,N)$ and
\begin{equation*}
  h(\boldsymbol{p})=n\deg P^{(l+1)}(z)+l=(N+k)n+l.
\end{equation*}
\end{proof}
Note that Lemma~\ref{lem:existence-solution} and (\ref{eq:M-height})
imply that there are $n$ vector polynomials in $\mathbb{S}(n,N)$ whose
heights are different elements of the factor space
$\integers/n\integers$. We will see later on that there are infinitely
many solutions for every equivalence class of the heights.

\begin{lemma}
\label{lem:form-height-S_n-minus-set-M}
Fix a natural number $m$ such that $1\le m<n$. If
$\boldsymbol{r}_1,\dots,\boldsymbol{r}_m$ are arbitrary elements of
$\mathbb{S}(n,N)$, then
$\mathbb{S}(n,N)\setminus[\mathbb{M}(\boldsymbol{r}_1)+\dots+\mathbb{M}(\boldsymbol{r}_m)]$
is not empty and
$h\left(\mathbb{S}(n,N)\setminus[\mathbb{M}(\boldsymbol{r}_1)+\dots+\mathbb{M}(\boldsymbol{r}_m)]\right)\neq
h(\boldsymbol{r_j})+nk$ for any $j\in\{1,\dots,m\}$ and
$k\in\mathbb{N}\cup\left\{0\right\}$. (In other words,
$h\left(\mathbb{S}(n,N)\setminus[\mathbb{M}(\boldsymbol{r}_1)+\dots+\mathbb{M}(\boldsymbol{r}_m)]\right)$
and $h(\boldsymbol{r_j})$ are different elements of the factor space
$\integers/n\integers$ for any $j\in\{1,\dots,m\}$).
\end{lemma}
\begin{proof}
  That
  $\mathbb{S}(n,N)\setminus[\mathbb{M}(\boldsymbol{r}_1)+\dots+\mathbb{M}(\boldsymbol{r}_m)]$
  is not empty follows from Lemma~\ref{lem:existence-solution} and
  (\ref{eq:M-height}) since $m<n$.
  We prove the second assertion by \emph{reductio ad absurdum}. Suppose
  that, for some $k_0\in\mathbb{N}\cup\left\{0\right\}$ and $j_0\in\{1,\dots,m\}$,
  \begin{equation*}
    h(\mathbb{S}(n,N)\setminus[\mathbb{M}(\boldsymbol{r}_1)+\dots+\mathbb{M}(\boldsymbol{r}_m)])=
  h(\boldsymbol{r}_{j_0})+nk_0\,.
  \end{equation*}
  Hence, there is
  $\boldsymbol{q}\in\mathbb{S}(n,N)\setminus[\mathbb{M}(\boldsymbol{r}_1)
  +\dots+\mathbb{M}(\boldsymbol{r}_m)]$ for which
  $h(\boldsymbol{q})=h(\boldsymbol{r}_{j_0})+nk_0$. Let
  $\boldsymbol{p}\in\mathbb{M}(\boldsymbol{r}_{j_0})$ such that
  $h(\boldsymbol{p})=h(\boldsymbol{r}_{j_0})+nk_0$. Then, by Lemma
  \ref{lem:height-properties}(\ref{lem:height-properties-c}), there is
  $c\in\mathbb{C}$ such that $h(\boldsymbol{q}+c\boldsymbol{p})\leq
  h(\boldsymbol{q})-1$. Clearly,
  $\boldsymbol{q}+c\boldsymbol{p}\in\mathbb{S}(n,N)$ but not in
  $\mathbb{M}(\boldsymbol{r}_1)+\dots+\mathbb{M}(\boldsymbol{r}_m)$. This
  contradicts the fact that $\boldsymbol{q}$ is an element of minimal height in
  $\mathbb{S}(n,N)\setminus[\mathbb{M}(\boldsymbol{r}_1)
  +\dots+\mathbb{M}(\boldsymbol{r}_m)]$.
\end{proof}

\begin{definition}
\label{def:j-generator}
Taking into account Definition~\ref{def:0-generator} and
Lemma~\ref{lem:form-height-S_n-minus-set-M}, for $1<j\le n$, one
defines recursively the $j$-th generator of $\mathbb{S}(n,N)$ as the
vector polynomial $\boldsymbol{r}_j$ in
$\mathbb{S}(n,N)\setminus[\mathbb{M}_1\dotplus\dots\dotplus\mathbb{M}_{j-1}]$
such that
\begin{equation*}
  h(\boldsymbol{r}_j)=
h(\mathbb{S}(n,N)
\setminus[\mathbb{M}_1\dotplus\dots\dotplus\mathbb{M}_{j-1}])
\end{equation*}
and $\mathbb{M}_j:=\mathbb{M}(\boldsymbol{r}_j)$.
\end{definition}
In this definition we have used direct sum ($\dotplus$) since
$\mathbb{M}_k\cap\mathbb{M}_l=\{0\}$ for $k\ne l$. This follows from
the fact that the nonzero vector polynomials in $\mathbb{M}_k$ and
$\mathbb{M}_l$ have different heights as a consequence of
(\ref{eq:M-height}) and
Lemma~\ref{lem:form-height-S_n-minus-set-M}. Clearly, each iteration
of this definition, up to $j=n$, makes sense as a consequence of
Lemma~\ref{lem:form-height-S_n-minus-set-M}.  Note also that
$\mathbb{M}_1\dotplus\dots\dotplus\mathbb{M}_j$ does not depend on the
choice of the $j$-th generator. Indeed, if, along with
$\boldsymbol{r}_j$, the vector polynomial $\boldsymbol{q}$ is a $j$-th
generator and $\boldsymbol{q}$ is not in
$\mathbb{M}_1\dotplus\dots\dotplus\mathbb{M}_j$, then, taking into
account that $h(\mathbb{S}
\setminus[\mathbb{M}_1\dotplus\dots\dotplus\mathbb{M}_{j-1}])$ is not
greater than $h(\mathbb{S}
\setminus[\mathbb{M}_1\dotplus\dots\dotplus\mathbb{M}_{j}])$, it is
straightforward to verify that
Lemma~\ref{lem:form-height-S_n-minus-set-M} yields a contradiction.

\section{Characterization of the solutions}
\label{sec:characterization-solutions}
This section deals with the properties of the generators of the
interpolation problem given in (\ref{eq:linear-combination}). By
elucidating the generators' properties, we are able to give a complete
description of all solution of the interpolation problem.

\begin{remark}
  \label{rem:heights-cover-factor-space}
 Due to Lemma~\ref{lem:form-height-S_n-minus-set-M}, and
 Definitions~\ref{def:0-generator} and \ref{def:j-generator}, one
 immediately obtains that
  \begin{equation*}
    \integers/n\integers=\{h(\boldsymbol{r}_1),\dots,h(\boldsymbol{r}_n)\}.
  \end{equation*}
\end{remark}

The following simple assertion is used to prove
Theorem~\ref{thm:estimation-formula} which gives estimates for the sum
of the heights of generators.

\begin{lemma}
  \label{lem:independence-of-generators}
  Let $n>1$. For $j\in\{1,\dots,n\}$, let $\boldsymbol{r}_j$ be the $j$-th
  generator of $\mathbb{S}(n,N)$. Then, there are infinitely many
  complex numbers $z$ such
  that the vectors $\boldsymbol{r}_1(z),\dots,\boldsymbol{r}_n(z)$
  in $\complex^n$ are linearly independent.
\end{lemma}
\begin{proof}
  We shall prove the lemma by \emph{reductio ad absurdum}. By
  continuity, if the vectors
  $\boldsymbol{r}_1(z),\dots,\boldsymbol{r}_n(z)$ are linearly
  dependent everywhere but a finite set of points, then they are
  linearly dependent everywhere. Suppose, $k\in\{2,\dots,n\}$ is the
  number for which the vector $\boldsymbol{r}_k(z)$ is a linear
  combination of $\boldsymbol{r}_1(z),\dots,\boldsymbol{r}_{k-1}(z)$
  for every $z$, but
  $\boldsymbol{r}_1(z_0),\dots,\boldsymbol{r}_{k-1}(z_0)$ are still
  linearly independent for a certain $z_0$. The latter means that
  \begin{equation*}
    \rank \left( \boldsymbol{r}_1(z_0) \, \dots \, \boldsymbol{r}_{k-1}(z_0)\right) = k-1\,.
  \end{equation*}
  By continuity, this rank also equals $k-1$ in some neighborhood of
  $z_0$. Also, by the hypothesis, for any $z\in\complex$,
  \begin{equation*}
    \boldsymbol{r}_k(z)=\sum_{l=1}^{k-1}F_l(z)\boldsymbol{r}_l(z)\,.
  \end{equation*}
  Solving this linear system for the unknown $F_l(z)$, the rank of the
  matrix being equal to $k-1$, we see that for any
  $l\in\{1,\dots,k-1\}$, $F_l$ is a rational function of
  $z$. Therefore there are scalar polynomials $S_0,\dots,S_{k-1}$ such
  that
 \begin{equation*}
   S_0(z)\boldsymbol{r}_k(z)=\sum_{l=1}^{k-1}S_l(z)\boldsymbol{r}_l(z)\,.
 \end{equation*}
 By Definition~\ref{def:j-generator}, taking into account
 (\ref{eq:M-height}) and Lemma~\ref{lem:form-height-S_n-minus-set-M},
 one concludes that all $S_l\boldsymbol{r}_l$ ($l\in\{1,\dots,k-1\}$)
 have different heights. Hence, by
 Lemma~\ref{lem:height-properties}(\ref{lem:height-properties-a})
 there is $l_0\in\{1,\dots,k-1\}$ such that
 \begin{equation}
   \label{eq:height-contradiction}
   h(S_0\boldsymbol{r}_k)=h(S_{l_0}\boldsymbol{r}_{l_0})\,,
 \end{equation}
 but according to (\ref{eq:M-height}) and
 Lemma~\ref{lem:form-height-S_n-minus-set-M} the r.\,h.\,s and the
 l.\,h.\,s of (\ref{eq:height-contradiction}) are different elements
 of $\integers/n\integers$. This contradiction finishes the proof.
\end{proof}
\begin{theorem}
  \label{thm:estimation-formula}
  Let $n>1$. For $l\in\{1,\dots,n\}$, let $\boldsymbol{r}_l$ be the $l$-th
  generator of $\mathbb{S}(n,N)$. Then, for any $m\in\{1,\dots,n\}$,
  \begin{equation*}
    \sum_{l=1}^mh(\boldsymbol{r}_l)\le Nm+\frac{m(m-1)}{2}\,.
  \end{equation*}
 In particular, when $m=n$, this gives an estimate of the sum of all generators.
\end{theorem}
\begin{proof}
  The proof is carried out inductively with respect to $m$. For $m=1$,
  the assertion has already been proven in
  Theorem~\ref{thm:height-bound-S_n}.\\[3mm]
Step 1. ($m=1\leadsto m=2$)

Let $N_1$ be an integer such that
\begin{equation}
  \label{eq:def-N0}
  0\le N_1\le N\,\quad\text{and}\quad h(\boldsymbol{r}_1)=N-N_1\,.
\end{equation}
Since $\boldsymbol{r}_1$ is a solution of
minimal height of the interpolation problem given by
$\{\sigma_j\}_{j=1}^N$ and $\{z_j\}_{j=1}^N$, it turns out that
$\boldsymbol{r}_1(\widetilde{z})$ does not vanish for any
$\widetilde{z}\ne z_1,\dots,z_N$. Indeed, otherwise
$\boldsymbol{r}_1(z)/(z-\widetilde{z})$ would be a solution of the
interpolation problem whose height is less than
$h(\boldsymbol{r}_1)$. Thus, choose the numbers
$z_{N+1},\dots,z_{N+N_1+1}$ each one of which is not equal to
$z_1,\dots,z_N$, and consider the interpolation problem
$\{\sigma_j\}_{j=1}^{N+N_1 +1}$ and $\{z_j\}_{j=1}^{N+N_1+1}$, where
the new matrices $\sigma_j$ are given by
\begin{equation}
  \label{eq:new-sigmas}
  \sigma_j:=
\left(
 \begin{array}{ccccc}
   \left|R_1(z_j)\right|^2&R_1(z_j)\overline{R_2(z_j)}&R_1(z_j)\overline{R_3(z_j)}&\ldots&R_1(z_j)\overline{R_n(z_j)}\\
R_2(z_j)\overline{R_1(z_j)}&\left|R_2(z_j)\right|^2&R_2(z_j)\overline{R_3(z_j)}&\ldots&R_2(z_j)\overline{R_n(z_j)}\\
R_3(z_j)\overline{R_1(z_j)}&R_3(z_j)\overline{R_2(z_j)}&\left|R_3(z_j)\right|^2&\ldots&R_3(z_j)\overline{R_n(z_j)}\\
   \vdots&\vdots&\vdots&\ddots&\vdots\\
R_n(z_j)\overline{R_1(z_j)}&R_n(z_j)\overline{R_2(z_j)}&R_n(z_j)\overline{R_3(z_j)}&\ldots&\left|R_n(z_j)\right|^2
 \end{array}\right)
\end{equation}
for $j=N+1,\dots,N+N_1+1$. Here the notation
$\boldsymbol{r}_1(z)=(R_1(z),\dots,R_n(z))^t$ has been used. According
to Theorem~\ref{thm:height-bound-S_n}, there is
$\boldsymbol{r}$ in $\mathbb{S}(\{\sigma_j\}_{j=1}^{N+N_1
  +1},\{z_j\}_{j=1}^{N+N_1+1})$ such that
\begin{equation}
  \label{eq:secon-gen-aux}
  h(\boldsymbol{r})\le N+N_1+1\,.
\end{equation}
Let us show that $\boldsymbol{r}\not\in\mathbb{M}_1$. To this end,
suppose on the contrary that
$\boldsymbol{r}(z)=S(z)\boldsymbol{r}_1(z)$ for some nonzero scalar
polynomial $S$. Taking into account (\ref{eq:new-sigmas}), it is
straightforward to verify that, for $j=N+1\,\dots,N+N_1+1$, one has
\begin{equation*}
  \inner{S(z_j)\boldsymbol{r}_1(z_j)}
  {\sigma_jS(z_j)\boldsymbol{r}_1(z_j)}=\abs{S(z_j)}^2
  \left(|R_1(z_j)|^2+\dots+|R_n(z_j)|^2\right)^2\,.
\end{equation*}
Therefore, by the way the
nodes $\{z_j\}_{j=N+1}^{N+N_1+1}$ have been chosen, the following
should hold
\begin{equation*}
  S(z_{N+1})=\dots=S(z_{N+N_1+1})=0\,.
\end{equation*}
Thus, $\deg S\ge N_1+1$. This inequality together with
(\ref{eq:height-Sp}) and (\ref{eq:def-N0}) imply that
$h(\boldsymbol{r})\ge N-N_1+n(N_1+1)$ which contradicts
(\ref{eq:secon-gen-aux}). Finally, observe that $\boldsymbol{r}$ is
in $\mathbb{S}(\{\sigma_j\}_{j=1}^{N},\{z_j\}_{j=1}^{N})$.\\[3mm]
Step 2. ($m>1\leadsto m+1\le n$)

Since the assertion is assumed to be proven for $m>1$, one can define
recursively the numbers $N_1,\dots,N_m$ such that for any
$l\in\{1,\dots,m\}$, the following holds
\begin{equation}
  \label{eq:def-N-k}
  h(\boldsymbol{r}_l)=N+(l-1)+N_{l-1}-N_l\,,
\end{equation}
where it is assumed that $N_0=0$.  We shall prove that there is a
vector polynomial $\boldsymbol{r}$ in
$\mathbb{S}(\{\sigma_j\}_{j=1}^{N},\{z_j\}_{j=1}^{N})$ such that
\begin{equation*}
   h(\boldsymbol{r})\le N+m+N_m\quad\text{ and }\quad
   \boldsymbol{r}\not\in\mathbb{M}_1\dotplus\dots\dotplus
     \mathbb{M}_{m}\,.
\end{equation*}
From this, the assertion of the theorem clearly will follow.

Consider set $I:=\{N+1,\dots,N+m+N_m\}$ and the sets
\begin{align*}
  I_1:&=\{N+1,\dots,N+L_1+1\}\,,\\
  I_2:&=\{N+L_1+2,\dots,
  N+L_2+2\}\,,\\
  \vdots\phantom{b:}&\phantom{=bbbc}\vdots\\
  I_m:&=\{N+L_{m-1}+m,\dots,N+m+N_m\}\,,
\end{align*}
where
\begin{equation*}
  L_k:=\sum_{j=1}^k\floor{\frac{m-j+1+N_m+N_j-N_{j-1}}{n}}\,.
\end{equation*}
Here $\floor{\cdot}$ is the floor function and it is again assumed that
$N_0=0$. Thus, $\{I_j\}_{j=1}^m$ is a partition of $I$, i.\,e.,
\begin{equation*}
  I=\bigcup_{j=1}^mI_j\quad\text{ and }\quad j\ne l\implies
  I_j\cap I_l=\emptyset\,.
\end{equation*}

Let $z_0\in\complex$ be such that the vectors
$\boldsymbol{r}_1(z_0),\dots,\boldsymbol{r}_m(z_0)$ are linearly
independent. The existence of such number is provided by
Lemma~\ref{lem:independence-of-generators}. Since the entries of each
of those vectors are polynomials, the vectors
$\boldsymbol{r}_1(z),\dots,\boldsymbol{r}_m(z)$ are also linearly
independent for any $z$ in a neighborhood of $z_0$. Take the points
$z_{N+1},\dots,z_{N+m+N_m}$ in this neighborhood such that
\begin{equation*}
  \{z_{N+1},\dots,z_{N+m+N_m}\}\cap\{z_{1},\dots,z_{N}\}=\emptyset\,,
\end{equation*}
and define the vectors
$\boldsymbol{\alpha}(j)=(\alpha_1(j),\dots,\alpha_n(j))^t$ in such a way that,
for each $l\in\{1,\dots,m\}$,
\begin{equation}
  \label{eq:conditions-on-new-alphas}
  \inner{\boldsymbol{\alpha}(j)}{\boldsymbol{r}_l(z_j)}\ne0
  \quad\text{ and }\quad\inner{\boldsymbol{\alpha}(j)}{\boldsymbol{r}_k(z_j)}=0
\end{equation}
for $k$ in $\{1,\dots,m\}\setminus\{l\}$ and $j\in I_l$. Note that the
linear independence of the vectors $\boldsymbol{r}_l(z_j)$ for any $l$
in $\{1,\dots,m\}$ and $j\in I$ guarantees the existence of
$\boldsymbol{\alpha}(j)$, $j\in I$, with the required properties.

For $j\in I$, define the matrices $\sigma_j$ using
(\ref{eq:sigma-matrix}) with the numbers $\alpha_k(j)$ given above and
consider the interpolation problem given by
$\{\sigma_j\}_{j=1}^{N+m+N_m}$ and $\{z_j\}_{j=1}^{N+m+N_m}$. By
Theorem~\ref{thm:height-bound-S_n}, there is $\boldsymbol{r}$ in
$\mathbb{S}(\{\sigma_j\}_{j=1}^{N+m+N_m},\{z_j\}_{j=1}^{N+m+N_m})$
such that $h(\boldsymbol{r})$ is not greater than $N+m+N_m$. It turns
out that $\boldsymbol{r}$ is not in
$\mathbb{M}_1\dotplus\dots\dotplus\mathbb{M}_m$, because if one
assumes
\begin{equation}
  \label{eq:wrong-assumption}
  \boldsymbol{r}(z)=\sum_{k=1}^mS_k(z)\boldsymbol{r}_k(z)
\end{equation}
with $S_k(z)$ being a scalar polynomial ($k\in\{1,\dots,m\}$), a
contradiction will follow. Indeed, one verifies from
(\ref{eq:conditions-on-new-alphas}) and (\ref{eq:wrong-assumption})
that
\begin{equation*}
 \inner{\boldsymbol{r}(z_j)}{\sigma_j\boldsymbol{r}(z_j)}=0\quad\text{
   for }\quad j\in I\,,
\end{equation*}
implies that
\begin{equation}
  \label{eq:degree-greater}
  \text{either }\quad S_l\equiv 0\quad\text{or}\quad\deg S_l\ge L_l-L_{l-1}+1\quad (L_0=0)
\end{equation}
for $l\in\{1,\dots,m-1\}$ and
\begin{equation*}
  \text{either }\quad S_m\equiv 0\quad\text{or}\quad\deg S_m\ge N_m-L_{m-1}+1\,.
\end{equation*}
On the other hand, taking into account
(\ref{eq:height-Sp}), Lemma~\ref{lem:height-properties}(a), and
Lemma~\ref{lem:form-height-S_n-minus-set-M}, one obtains
after straightforward calculations that
\begin{equation}
  \label{eq:degree-less}
  \deg S_l\le L_l-L_{l-1}\quad\text{ for } l\in\{1,\dots,m\}\,.
\end{equation}
It follows from (\ref{eq:degree-greater}) and (\ref{eq:degree-less})
that
\begin{equation*}
  S_1(z)\equiv\dots\equiv S_{m-1}(z)\equiv 0\,.
\end{equation*}
Analogously, to prove that $S_m\equiv 0$, one shows that
$\deg S_m\ge N_m-L_m+1$ is incompatible with
(\ref{eq:degree-less}) for $l=m$. This is done by verifying that
\begin{equation}
\label{eq:inequality-to-prove}
  N_m-L_m+1>0\,.
\end{equation}
In view of (\ref{eq:def-N-k}) and
Lemma~\ref{lem:form-height-S_n-minus-set-M}, the numbers
$m-j+1+N_m+N_j-N_{j-1}$, for $j\in\{1,\dots,m\}$, are different
elements of the space $\integers/n\integers$. Therefore there exists a permutation
$\{a_1,\dots,a_n\}$ of $\{0,\dots,n-1\}$, such that
\begin{equation*}
  \floor{\frac{m-j+1+N_m+N_j-N_{j-1}}{n}}=\frac{m-j+1+N_m+N_j-N_{j-1}-a_j}{n}\,.
\end{equation*}
There is at
most one $j$ in $\{1,\dots,m\}$ such that $a_j=0$. Moreover,
\begin{equation*}
  \sum_{j=1}^ma_j\ge\sum_{j=1}^{m}(j-1)\,.
\end{equation*}
Thus, the l.\,h.\,s. of the inequality (\ref{eq:inequality-to-prove})
can be rewritten as follows
\begin{equation*}
  N_m-\sum_{j=1}^m\frac{m-j+1+N_m+N_j-N_{j-1}-a_j}{n} +1\,.
\end{equation*}
Since
\begin{equation*}
  \sum_{j=1}^m\left(m-j+1-a_j\right)\le m\,,
\end{equation*}
one has
\begin{align*}
  N_m-&\sum_{j=1}^m\frac{m-j+1+N_m+N_j-N_{j-1}-a_j}{n} +1\\ &\ge
  N_m-\sum_{j=1}^m\frac{m+N_m+N_j-N_{j-1}}{n}+1\\
  &\ge N_m-\frac{m+(m+1)N_m}{n}+1>0\,.
\end{align*}
In the last inequality, it has been used that $m+1\le n$.
\end{proof}
\begin{theorem}
  \label{thm:conjecture-equality}
  Let $\boldsymbol{r}_j$ be the $j$-th generator of
  $\mathbb{S}(n,N)$. It holds true that
  \begin{equation*}
    \sum_{j=1}^nh(\boldsymbol{r}_j)=Nn+\frac{n(n-1)}{2}\,.
  \end{equation*}
\end{theorem}
\begin{proof}
  Because of Theorem~\ref{thm:estimation-formula}, it suffices to show
  that
  \begin{equation}
    \label{eq:second-inequality-conjecture}
    \sum_{j=1}^nh(\boldsymbol{r}_j)\ge Nn+\frac{n(n-1)}{2}\,.
  \end{equation}
  Suppose that this is not true and define
  \begin{equation*}
    Q(z):=\det\left(\boldsymbol{r}_1(z)\dots\boldsymbol{r}_n(z)\right)\,,
  \end{equation*}
  where $\left(\boldsymbol{r}_1(z)\dots\boldsymbol{r}_n(z)\right)$ is
  the square matrix with columns given by the vectors
  $\boldsymbol{r}_1(z),\dots,\boldsymbol{r}_n(z)$. For the entries of
  the generators, we use the notation
  \begin{equation*}
    \boldsymbol{r}_j(z)=\left(R_1^{(j)},\dots,R_n^{(j)}\right)^t\quad\forall
    j\in\{1,\dots,n\}\,.
  \end{equation*}
  It follows from Definition~\ref{def:height-polyn-vect}, that for any
  $j$ in $\{1,\dots,n\}$, there is $l(j)\in\{1,\dots,n\} $ such that
  \begin{equation*}
    h(\boldsymbol{r}_j)=n\deg R_{l(j)}^{(j)}+l(j)-1\,.
  \end{equation*}
  Moreover, by Lemma~\ref{lem:form-height-S_n-minus-set-M}, when $j$
  runs through the set $\{1,\dots,n\}$, $l(j)$ also runs through
  $\{1,\dots,n\}$. Therefore,
$$
    \sum_{j=1}^nh(\boldsymbol{r}_j)=n\sum_{j=1}^n \deg R_{l(j)}^{(j)}
    +\sum_{j=1}^n(j-1)\,.
 $$
 Thus the negation of (\ref{eq:second-inequality-conjecture})  imply that
 \begin{equation}
   \label{eq:sum-degrees-conjecture}
   \sum_{j=1}^n \deg R_{l(j)}^{(j)}<N\,.
 \end{equation}
  On the other hand, since the interchanging of two columns of a matrix
  leads to multiplying the corresponding determinant by $-1$, it is
  clear that for calculating the degree of the polynomial $Q$, one
  could use any arrangement of the vectors $\boldsymbol{r}_j$,
  $j\in\{1,\dots,n\}$. Thus,
  \begin{equation*}
    \deg Q(z)= \deg\det
    \left(\boldsymbol{r}_{l^{-1}(1)}(z)
      \dots\boldsymbol{r}_{l^{-1}(n)}(z)\right)\,.
  \end{equation*}
  Note that in this arrangement of the columns the diagonal elements of
  the matrix are the polynomials that determine the height of the
  generators. Hence, it is straightforward to verify that $\deg Q$ is
  the sum of the degree of the diagonal elements of
  $\left(\boldsymbol{r}_{l^{-1}(1)}(z)
    \dots\boldsymbol{r}_{l^{-1}(n)}(z)\right)$, that is,
 \begin{equation}
   \label{eq:degree-of-Q}
   \deg Q(z)=\sum_{j=1}^n \deg R_{l(j)}^{(j)}\,,
 \end{equation}
and, by (\ref{eq:sum-degrees-conjecture}), this is $<N$.

Now, fix a node of interpolation $z_{l_0}$ and observe that, since
$\boldsymbol{r}_j$ is in $\mathbb{S}(n,N)$ for all
$j\in\{1,\dots,n\}$, one has
\begin{equation*}
  \sum_{k=1}^n\alpha_k(l_0)R_k^{(j)}(z_{l_0})=0,\qquad j\in\{1,\dots,n\}\,.
\end{equation*}
By construction, this system has a solution, that is, the determinant
of the system vanishes, so $Q(z_{l_0})=0$. Since the interpolation
node $z_{l_0}$ was arbitrary, one concludes that
\begin{equation*}
  Q(z_1)=\dots=Q(z_N)=0\,.
\end{equation*}
These equalities, together with (\ref{eq:sum-degrees-conjecture}) and
(\ref{eq:degree-of-Q}), imply that $Q(z)\equiv 0$ which contradicts Lemma 
\ref{lem:independence-of-generators}.
\end{proof}
\begin{theorem}
\label{thm:all-solution-comb-generators}
Let $n\ge 2$. Any element $\boldsymbol{p}$ of $\mathbb{S}(n,N)$, can be
written in the form
\begin{align*}
\boldsymbol{p}=\sum_{j=1}^{n}S_j\boldsymbol{r}_j\,,
\end{align*}
where $S_j$ is a scalar polynomial and $\boldsymbol{r}_j$ is the $j$-th
generator of the interpolation problem ($j\in\{1,\dots,n\}$).
\end{theorem}
\begin{proof}
For $j\in\{2,\dots,n\}$, consider the sets
\begin{equation}
  \label{eq:sets}
  \begin{split}
  \mathcal{B}_j:=&\left\{m\in\mathbb{N}:
    \: m=h(\boldsymbol{r}_{k})+nl+1, \text{ where } k<j\text{ and }
    l\in\mathbb{N}\cup\left\{0\right\}\right\}\,,\\
  \mathcal{A}_1:=&\left\{1,2,\dots,h(\boldsymbol{r}_1)\right\}\,,\\
  \mathcal{A}_j:=&\left\{h(\boldsymbol{r}_{j-1})+1,\dots,h(\boldsymbol{r}_j)\right\}\setminus\mathcal{B}_j\,.
\end{split}
\end{equation}
Now, define the  sequence
$\left\{\boldsymbol{g}_k\right\}_{k\in\mathbb{N}}$ as follows
\begin{equation*}
  \boldsymbol{g}_k(z):=
  \begin{cases}
    \boldsymbol{e}_k(z) & \text{for}\ k\in\cup_{j=1}^n\mathcal{A}_j\,,\\
    z^l\boldsymbol{r}_j(z) & \text{for}\ k=h(\boldsymbol{r}_{j})+nl+1\,,
  \end{cases}
\end{equation*}
where $\boldsymbol{e}_k$ is given in (\ref{eq:basis}). Note that
$h(\boldsymbol{g}_k)=k-1$. Therefore, by Theorem
\ref{thm:basis-height}, we have that
$\left\{\boldsymbol{g}_k\right\}_{k\in\mathbb{N}}$ is a basis in
$\mathbb{P}$. So, any $\boldsymbol{p}\in\mathbb{S}(n,N)$ can be
written as
\begin{align}
  \boldsymbol{p}&=\sum_{j\in\nats}c_j\boldsymbol{g}_j\nonumber\\
  &=\sum_{j=1}^n\sum_{k\in\mathcal{A}_j}c_{k}\boldsymbol{e}_{k}+\sum_{j=1}^{n}S_j\boldsymbol{r}_j\,,
\label{eq:=sum_j=0k_0-1c_j-k_j}
\end{align}
where $S_j$ is a scalar polynomial. Since $\boldsymbol{p}$ and
$\sum_{j=0}^{n-1}S_j\boldsymbol{r}_j$ are in $\mathbb{S}(n,N)$,
one has that
\begin{equation}
  \label{eq:found-solution}
  \sum_{j=1}^n\sum_{k\in\mathcal{A}_j}c_{k}\boldsymbol{e}_{k}
\end{equation}
is in $\mathbb{S}(n,N)$.

Let us show that (\ref{eq:found-solution}) is a trivial solution. 
Suppose on the contrary that (\ref{eq:found-solution}) is
nontrivial, i.\,e., there is $k\in\cup_{j=1}^n\mathcal{A}_j$
such that $c_k\ne 0$. Let
\begin{equation*}
  k_0:=\max_{1\le j\le n}\{k\in\mathcal{A}_j: c_k\ne0\}
\end{equation*}
and $\mathcal{A}_{j_0}$ be such that $k_0\in\mathcal{A}_{j_0}$.
By Lemma~\ref{lem:height-properties}(a) the height of
(\ref{eq:found-solution}) is equal to $k_0-1$. Therefore, by (\ref{eq:sets}),
\begin{equation*}
  h\left(\sum_{j=1}^n\sum_{k\in\mathcal{A}_j}c_{k}\boldsymbol{e}_{k}\right)
  <h(\boldsymbol{r}_{j_0})
\end{equation*}
and, by construction,
\begin{equation*}
  \sum_{j=1}^n\sum_{k\in\mathcal{A}_j}c_{k}\boldsymbol{e}_{k}\not\in
  \mathbb{M}_1\dotplus\dots\dotplus\mathbb{M}_{j_0-1}
\end{equation*}
which contradicts Lemma~\ref{lem:form-height-S_n-minus-set-M}.
\end{proof}
\def\cprime{$'$} \def\lfhook#1{\setbox0=\hbox{#1}{\ooalign{\hidewidth
  \lower1.5ex\hbox{'}\hidewidth\crcr\unhbox0}}} \def\cprime{$'$}
  \def\cprime{$'$} \def\cprime{$'$}

\end{document}